\documentclass[12pt]{article}
\usepackage[utf8]{inputenc}

\title{Approximation of Functionals by Neural Network without Curse of Dimensionality}
\author{Yahong Yang$^{a}$\thanks{E-mail address: \it\textbf{yyangct@connect.ust.hk}},
	Yang Xiang$^{a,b}$\thanks{E-mail address: \it\textbf{maxiang@ust.hk}}\\$^{a}$\small\textit{Department of Mathematics, Hong Kong University of Science and Technology,}\\ \small\textit{Clear Water Bay, Kowloon, Hong Kong}\\
	$^{b}$\small\textit{HKUST Shenzhen-Hong Kong
		Collaborative Innovation Research Institute,} \\ \small\textit{Futian, Shenzhen, China}}

\usepackage[numbers]{natbib}
\usepackage{amsmath}
\usepackage{epsfig,amsthm}
\usepackage{amssymb,amsfonts}
\usepackage{dsfont}
\usepackage{subfigure}
\usepackage{indentfirst}
\usepackage{mathrsfs}
\usepackage{color}
\usepackage{graphicx}
\usepackage{zymacros}
\usepackage{hyperref}

\begin{document}
		\maketitle

	\begin{abstract}
		In this paper, we establish a neural network to approximate functionals, which are maps from infinite dimensional spaces to finite dimensional spaces.
		The approximation error of the neural network is $O(1/\sqrt{m})$ where $m$ is the size of networks, which overcomes the curse of dimensionality. The key idea of the approximation is to define a Barron space of functionals.
	\end{abstract}

	\section{Introduction}

	Recently, neural networks have  revolutionized many fields of science and engineering including computational and applied mathematics. As one of the  important applications of neural networks in applied mathematics, many methods have been developed for solving partial differential equations (PDEs) by neural networks, e.g., \cite{Lagaris1998,weinan2017deep,weinan2018deep,han2018solving,Grohs2018,Sirignano2018,Khoo2019,raissi2019physics,shin2020convergence,luo2020two,lu2021priori}.
	The boundary value problem of a PDE in a domain $\Omega$ in $d$-dimensional space takes the form
	\begin{equation}
		\begin{cases}\fL u=g_1 & \text { in } \Omega, \\ \fA u=g_2 & \text { on } \partial \Omega,\end{cases}\label{PDE}
	\end{equation}
	where $u$ is the unknown function, $\fL$ is a partial differential operator, $\fA$ is the operator for specifying an appropriate boundary condition, $g_1$ and $g_2$ are given functions, and without loss of generality, $\Omega=[0,1]^d$. The key idea of using neural networks to solve PDEs is to obtain $u(\vx;\vtheta)$ from a neural network, where $\vtheta$ denotes all the parameters in the neural network that are trained by optimizing some loss function associated with the PDE. That is, in these methods, a neural network is established for the solution function $u(\vx)$ with given function pair $g_1$ and $g_2$;
	for a different pair of functions $g_1, g_2$ in Eq.~\eqref{PDE}, the $u(\vx;\vtheta)$ from the neural network has to be learned again although operators $\fL$ and $\fA$ are the same.
	
	The PDE boundary value problem in Eq.~\eqref{PDE} can be considered  as an operator $u(\vx)=G(g_1,g_2)$. If we can learn the operator $G$ directly by a neural network, we will be able to obtain the solution of  Eq.~\eqref{PDE} for any given function pair $g_1, g_2$ without learning again. 
	A few methods have been proposed for learning operators by neural networks for solving PDEs, such as DeepONet \cite{lu2019deeponet}, DeepGreen \cite{gin2020deepgreen}, Fourier Neural Operator (FNO) \cite{li2020fourier}, Neural Operator \cite{li2020multipole,li2020neural}, MOD-Net \cite{zhang2021mod}, and the deep learning-based nonparametric estimation \cite{liu2022deep}.
	The DeepGreen, Neural Operator and MOD-Net methods are based on Green's functions for solving PDEs, i.e., these methods learn the Green's function instead of learning the operator directly.
	Since in general only solutions of linear PDEs have the Green's function formulation, those Green's function based methods cannot be used directly to solve nonlinear PDEs, and accuracy of some proposed attempts for the extension of those Green's function based methods to nonlinear PDEs has not been rigorously proved in the literature.
	The DeepONet \cite{lu2019deeponet} is a method that learns nonlinear operators associated with PDEs from data  based on the approximation theorem for operators by neural networks \cite{chen1995universal}. The method in \cite{liu2022deep} is to learn the  operator by model reduction \cite{bhattacharya2020model} of reducing the operator to a finite dimensional space.
	Most of these available works on learning operators focused on the development of algorithms.
	
	The curse of dimensionality is a serious issue that generally exists for approximations in high dimensions. Note that for the space of functions as the domain of the operator associated with PDE boundary value problem in Eq.~\eqref{PDE}, the dimension is $\infty$.
	The curse of dimensionality \cite{Bellman1957curse} summarizes this property that in order to maintain the accuracy of an approximation, the number of sample points grows exponentially with the increase of dimension. This means that for a fixed number of sample points, the accuracy will be lost in an exponential way as the dimension increases. Only a few analyses have been performed on overcoming the curse of dimensionality \cite{de2021convergence,lanthaler2022error,liu2022deep}, and they all focused on reducing the infinite dimensional space to a low dimensional space. Exponential dependence on the dimension for the sample points still exists in these methods, which requires that the dimension of the reduced space has to be low enough. Moreover, the Bayesian inversion learning method in \cite{de2021convergence} is only for linear operators.
	

	Convergence of the DeepONet method \cite{lu2019deeponet} for operator learning is  guaranteed by the approximation theorem for operators by neural networks \cite{chen1995universal} (Theorem \ref{operator} \cite{chen1995universal} in Appendix), which, however, does not provide the accuracy dependence on the number of sample points.
	\cite{lanthaler2022error} studied the error of the DeepONet method and overcame the curse of dimensionality in this method by considering  smooth functions with exponentially decaying coefficients in the Fourier series.
	Here we perform a general error analysis for the DeepONet method in terms of the size of networks and the number of parameters. It also serves for the examination of sources of the curse of dimensionality in this method.
	
	\underline{\bf Error estimate for DeepONet}.
	In DeepONet, for a continuous  operator $G$ defined in $W_1^\infty([0,1]^{d_1})\to W_1^\infty([0,1]^{d_2})$ and a function $v\in W^\infty_1([0,1]^{d_1})$, $G(v)$ is a function belonging to $W_1^\infty([0,1]^{d_2})$. This $G(v)$ can be approximated by a linear combination of the activation function $\sigma(x)=\max\{x,0\}$ in the following form by Theorem \ref{function} \cite{mhaskar1996neural} (given in Appendix):
	\begin{equation}
		\left\|G(v)(\vy)-\sum_{k=1}^{p} c_{k}\left[G(v)\right] \sigma\left(\vw_{k}\cdot \vy+\zeta_k\right)\right\|_{L_\infty([0,1]^{d_2})} \leq C p^{-1 / d_2}\|G(v)\|_{W^{\infty}_{1}([0,1]^{d_2})},\label{step1}
	\end{equation}
	where $\vw_k\in\sR^{d_2}, \zeta_{k}\in\sR$ for $k=1, \ldots, p$, $c_k$ is continuous functionals, and $C$ (or $C_1$, $C_2$ in other inequalities in this paper) is some constant independent of the parameters.
	Denote $f_k(v):=c_{k}\left[G(v)\right]$, which is a functional from $W_1^\infty([0,1]^{d_1})\to \sR$. The remaining task in DeepONet is to approximate these functionals by neural networks.

	In DeepONet, the domain of functions $[0,1]^{d_1}$ is divided into $s^{d_1}$ parts by the $N:=(s+1)^{d_1}$ nodes denote by $\{\vx_j\}_{j=1}^N$, where each
	$x_j\in \left\{0,\frac{1}{s},\frac{2}{s},\cdots\cdots,\frac{s-1}{s},1\right\}^{d_1}$.
	Based on these $N$ nodes, a piece-wise linear function of $v_N(\vx)$ is defined in DeepONet such that $\|v_N(\vx)-v(\vx)\|_{\infty}\le C_0s^{-1}$.
	Further assume that $f_k$ is a Lipschitz continuous functional in $L_\infty([0,1]^{d_1})$ with a Lipschitz constant $L_k$, then we have
	\begin{equation}
		\left|f_k(v)-f_k(v_N)\right|\le L_k\|v_N(\vx)-v(\vx)\|_{\infty}\le C_1s^{-1}\le C_1N^{-1/d_1}.\label{deepstep2.1}
	\end{equation} For each $f_k(v_N)$, a function $h_k:\sR^N\to \sR$ is defined such that
	\begin{equation}
		h_k(v(\vx_1),v(\vx_2),\cdots\cdots,v(\vx_N)):=f_k(v_N).\label{g}
	\end{equation}
	Suppose that $v$ is bounded by $1$ and $h_k\in W^\infty_1([-1,1]^N)$. By Theorem \ref{function} \cite{mhaskar1996neural}, we have \begin{align}
		\left|h_k(v(\vx_1),v(\vx_2),\cdots\cdots,v(\vx_N))-\sum_{i=1}^{m} c_{i} \sigma\left(\sum_{j=1}^{N} \xi_{i j} v\left(\vx_{j}\right)+\theta_{i}\right)\right| \leq C_2 m^{-\frac{1}{N} }\|h_k\|_{W^\infty_1([-1,1]^N)}.\label{deepstep2.2}
	\end{align}
	Combining Eqs.~(\ref{step1}),(\ref{deepstep2.1}) and (\ref{deepstep2.2}), with several further assumptions, the error of DeepONet is bounded by \begin{align}
		&\left\|G(v)(\vy)-\sum_{k=1}^{p} \sum_{i=1}^{m} c_{i}^{k} \sigma\left(\sum_{j=1}^{N} \xi_{i j}^{k} v\left(\vx_{j}\right)+\theta_{i}^{k}\right){\sigma\left(\vw_{k} \cdot \vy+\zeta_{k}\right)}\right\|_{L_\infty([0,1]^{d_2})}\notag\\\le& C p^{-1 / d_2}\|G(v)\|_{W^{\infty}_{1}([0,1]^{d_2})}+C_1(p,L)N^{-1/d_1}+C_2(p,L,\lambda) m^{-1/{N} },\label{error}
	\end{align}
	where $L=\max_{1\le k\le p}L_k$ and $\lambda=\max_{1\le k\le p}\|h_k\|_{W^\infty_1([-1,1]^N)}$.
	
	It can be seen that every term in Eq.~(\ref{error}) has the problem of the curse of dimensionality. The first term in Eq.~(\ref{error}) is the error for approximating the operators by some functionals, after which the learning of operators is reduced to learning of functionals.
	The last two terms in Eq.~(\ref{error}) are errors from the approximation of functionals by neural networks (the second one from the approximation of the function by a piecewise linear function and the last one from the approximation of this piecewise linear function by a neural network), in which
	the curse of dimensionality exists even if we consider such an approximation in  low dimensional spaces (i.e., small $d_1$ and $d_2$).
	In fact, for the second term $C_1(p,L)N^{-1/d_1}\leq\varepsilon$, where $\varepsilon$ is small, we should have at least $N\sim1/\varepsilon^{d_1}$,  and accordingly, the last term $C_2(p,L,\lambda) m^{-1/{N}}\sim m^{-\varepsilon^{d_1}}$, whose curse of dimensionality is more serious than that of $m^{-\frac{1}{d_1}}$.

	
	In this paper, we focus on the approximation of functionals, i.e.,  maps from a space of functions which has infinite dimensions to $\sR$, by neural networks without curse of dimensionality. It can be seen from the above analysis that approximating functionals by neural networks is a
	crucial step in approximating operators by neural networks in the DeepONet and similar methods, whose curse of dimensionality is serious in these methods. Moreover, functionals, such as linear functionals and energy functionals, have a wide range of  important applications in science and engineering fields.
	There are only limited attempts in the literature on using  neural networks to approximate functionals~\cite{chen1993approximations,thind2020deep}. In \cite{chen1993approximations}, they reduced the approximation of functionals to the approximation of  functions by reducing the function space into a finite dimensional space (Theorem \ref{functional} \cite{chen1993approximations}), which was adopted in the second step of the DeepONet method, and as we discussed above, this treatment has a serious curse of dimensionality. In \cite{thind2020deep}, they approximated functionals directly by neural networks but did not show the error of their method.
	
	Several methods for overcoming the curse of dimensionality in approximations of functions have been developed in the literature. An approximation method was proposed by Barron \cite{barron1993universal} based on some function space with spectral norm; see also further developments of this approach \cite{pisier1981remarques,bach2017breaking,klusowski2018approximation,lu2021priori}. Another type of function spaces based on the neural network representation and probability in the parameter space were also introduced and developed \cite{weinan2019priori,weinan2019barron,wojtowytsch2020some,luo2020two,li2020complexity,ma2022barron}. Following \cite{weinan2019barron,ma2022barron}, in this paper, the former type of spaces are referred to as Barron spectral spaces, and the latter Barron spaces.
			The approximation error in a Barron/Barron spectral space is able to reduce to $O(1/\sqrt{m})$, where $m$ is the size of networks.  However, the curse of dimensionality in approximation functionals cannot be solved by these Barron/Barron spectral space methods directly. In fact, the domain of functionals is an infinite dimensional space which is essentially different from the space of functions. A naive idea of generalizing the Barron/Barron spectral space methods for functions to functionals is to approximate the infinite dimensional space of functions by some finite dimensional space. However, as we demonstrated above for the  DeepONet method, the approximation of an infinite dimensional space by  a finite dimensional space through sample points still has the problem of curse of dimensionality (the second error in Eq.~\eqref{error}), even though the curse of dimensionality in the last step in the  DeepONet method (the last error in Eq.~\eqref{error}) in principle can be overcome by some form of the Barron/Barron spectral space method for functions without curse of dimensionality.
			Furthermore, in such a straightforward generalization, we only know that the finite dimensional function $h_k$ in Eq.~(\ref{g}) exists, and it is not easy to check if $h_k$ belongs to the Barron/Barron spectral space.
			
			In this paper, we establish a new method for the approximation of functionals by neural networks without curse of dimensionality, which is based on Fourier series of functionals and the associated Barron spectral space of functionals.
			Specifically,   we first establish Fourier series of functionals, and then  prove that any functional satisfying some proper assumption (Assumption \ref{assump}) can be approximated by these Fourier series (Theorem \ref{limit}). We then define a Barron spectral space $\fB_s$ and a Hilbert space $\fH_{s}$ of the functionals (Definition \ref{space}) based on the Fourier series, and estimate the error of the approximating neural network based on the Barron spectral space (Theorem \ref{main}). The approximation error of the neural network is $O(1/\sqrt{m})$ where $m$ is the size of networks, which overcomes the curse of dimensionality.
			Under some stronger conditions (including smoothness of the functions in the domain of the functional), a simpler method for learning functionals by neural networks has been proved (Theorem \ref{cutoff1}).
			Applications of these obtained theorems on the approximation of functionals by neural networks including the application to solving PDEs by neural networks are discussed.  
			
			
			Our main contributions are:
			
			$\bullet$ We establish a new method for the approximation of functionals by a neural network, by defining  (i) a Fourier-type series in the infinite-dimensional space of functionals and (ii) the associated spectral Barron spectral space $\fB_s$ and a Hilbert space $\fH_{s}$ of functionals. We show that the proposed  method for the approximation of functionals overcomes the curse of dimensionality.
			

			$\bullet$ The established method for approximation of functionals without curse of dimensionality can be employed in learning functionals, such as linear functionals and energy functionals in science and engineering fields. It can also be used to solve PDE problems by neural networks at  some given points. This method provides a basis for the further development of methods for learning operators.
			
			

			\section{Fourier Series and Barron Spectral Space of Functionals}
			In this section, we define a Barron spectral space of functionals. For this purpose,
			we first define a Fourier series in the infinite-dimensional space of functionals. These definitions associated with the space of functionals are based on a basis of the domain of the functionals (which is a function space).

			First of all, we define the index set of a basis of the space of functionals.
			\begin{defi}[Index set of basis of functionals] $\sK\subset\sZ^\infty$ is defined by
				\begin{flalign}
					\sK:=\bigcup\limits_{m=1}^\infty\sK^m,
					\label{sk}
				\end{flalign}
				where
				$\sK^m:=\big\{\vk:=(k_1,k_2,\cdots,k_i,\cdots)\in\sZ^\infty: N_{\vk}:=\max\{i:k_i\not=0\}\le m\big\}$.
			\end{defi}

			Due to the definition of $\sK$, we know that $\vk:=(k_1,k_2,\cdots,k_i,\cdots)\in\sK$ has only finite nonzero $k_i$. Furthermore, $\sK$ is a countable set due to Axiom of Choice; see Lemma \ref{Cardinality} below. Note that all the proofs in this paper including that of this lemma are given in Appendix.

			\begin{lem}\label{Cardinality}
				$\sK$ is a countable set.
			\end{lem}
		
		\begin{proof}[Proof of Lemma \ref{Cardinality}]
			Denote the cardinality of a set $\sA$ as $|\sA|$ and $|\sN|=\aleph_0$. Hence \begin{equation}
				|\sK^m|=|\{\vk:=(k_1,k_2,\cdots,k_i,\cdots)\in\sN^\infty: N_{\vk}:=\max\{i:k_i\not=0\}\le m\}|=|\sN^m|=\aleph_0.
			\end{equation}
			
			Since a countable union of countable sets is countable (by using the Axiom of Choice), we have
			\begin{equation}
				|\sK|=\left|\bigcup\limits_{m=1}^\infty\sK^m\right|=\aleph_0.
			\end{equation}
		\end{proof}
			
			
			Denote the domain of functionals as $\Omega$, which is a Banach space of functions. Let $\{\Phi_i\}_{i=1}^\infty$ be a Schauder basis of $\Omega$, i.e., for every element $v\in\Omega$, there exists a unique sequence $\{b_i\}_{i=1}^\infty$ of scalars in $\sR$ such that $v=\sum_{i=1}^\infty b_i\Phi_i$. For example,  $\{\Phi_n\}_{n=1}^\infty=\{\sqrt{2} \sin (2 \pi n x) \mid n \in \mathbb{N}\} \cup\{\sqrt{2} \cos (2 \pi n x) \mid n \in \mathbb{N}\} \cup\{1\}$  when $\Omega=L_2(0,1)$.
			Note that we will define a Barron spectral space of functionals based  only on the sequence $\{b_i\}_{i=1}^\infty$, and the defined Barron spectral space of functionals will apply to any basis  $\{\Phi_i\}_{i=1}^\infty$ and any convergence of $v=\sum_{i=1}^\infty b_i\Phi_i$ as long as  $\{b_i\}_{i=1}^\infty$   is uniquely determined.
			Next, we define  a basis of the functionals on $\Omega$.
			\begin{defi}[Fourier basis of functionals]\label{fourier}For any $v=\sum_{i=1}^\infty b_i\Phi_i\in L_\text{bound}(\Omega)$, where \begin{equation}
					L_\text{bound}(\Omega):=\left\{v=\sum_{i=1}^\infty b_i\Phi_i\in\Omega: -\frac{1}{2}< b_i< \frac{1}{2}, i\in\sN_+\right\},
				\end{equation}
				the Fourier basis $\{e_{\vk}\}_{\vk\in\sK}$ of the functionals with domain $\Omega$ based on $\{\Phi_i\}_{i=1}^{\infty}$ is \begin{equation}
					e_{\vk}(v):=\prod_{i=1}^\infty\exp \left(2\pi \I  k_ib_i\right)=\exp \left(\sum_{i=1}^\infty 2\pi \I  k_ib_i\right).
				\end{equation}
			\end{defi}
			(The condition $-\frac{1}{2}< b_i< \frac{1}{2}$ is to restrict the area of the domain $\left(-\frac{1}{2},\frac{1}{2}\right)^{|\sD_j|}$ 
			to $1$.)
			
			We make the following assumptions for the functionals being considered:
			\begin{assump}[Finite dimensional summation]\label{assump}	
				For a basis of $\Omega$, $\{\Phi_i\}_{i=1}^\infty$, there is a unique sequence of set $\sD:=\{\sD_j\}_{j=1}^\infty$ where $\sD_j\subset \sN_+$ and $|\sD_j|<+\infty$, such that the functional $f$ satisfies
				\begin{align}
					f\left(\sum_{i=1}^\infty b_i\Phi_i\right)=\sum_{j=1}^\infty f_j\left(\sum_{i\in\sD_j} b_{i}\Phi_{i}\right),\label{divide}
				\end{align} for $v=\sum_{i=1}^\infty b_i\Phi_i\in L_\text{bound}(\Omega)$.
				For each $f_j\left(\sum_{i\in\sD_j} b_{i}\Phi_{i}\right)$, it cannot be further divided into $f_j^{(1)}\left(\sum_{i\in\sD_j^1} b_{i}\Phi_{i}\right)$ and $f_j^{(2)}\left(\sum_{i\in\sD_j^2} b_{i}\Phi_{i}\right)$ where $|\sD_j^1|,|\sD_j^2|< |\sD_j|$.
			\end{assump}

			 Note that Assumption \ref{assump} does not require $\sD_j\cap\sD_i=\emptyset$ for $i\not=j$. We will give examples of functionals that satisfy Assumption \ref{assump} and demonstrate its generality at the end of this section.

			Due to  Assumption \ref{assump}, the infinite-dimensional space of functions as the domain of the functional can be decomposed into finite-dimensional subspaces. As a result, in each finite-dimensional subspace, Fourier coefficients can be defined. 	Under this assumption, we can define the Fourier coefficients of a functional as follows.
			
			\begin{defi}[Fourier coefficients of a functional]\label{fourier coefficient}
				Suppose that Assumption \ref{assump} holds for the functional $f$, the Fourier coefficients of the functional $f$ with basis $\{e_{\vk}\}_{\vk\in\sK}$ are define as
				\begin{equation}				a_{\vk}(f):=\sum_{j\in\sA_{\vk}}\int_{\left(-\frac{1}{2},\frac{1}{2}\right)^{|\sD_j|}}f_j\left(\sum_{i\in\sD_j} b_{i}\Phi_{i}\right)\left[\prod_{i\in\sD_j}\exp \left(-2\pi \I  k_ib_i\right)\,\D b_i\right],\label{coefficient}
				\end{equation}where
				\begin{align}\label{eqn:Ak}
					\sA_{\vk}:=\{j\in\sN:\vk\in\sK_{\sD_j}\}, \ \ \
					\sK_{\sD_j}:=\{\vk:=(k_1,k_2,\cdots,k_i,\cdots)\in\sK:k_i=0 \ {\rm for} \ i\notin\sD_j\}.
				\end{align}
			\end{defi}
			Note that the summation over $\sA_{\vk}$ is due to the fact that each $\vk$ may be associated with multiple $D_j$'s.


			By Definition \ref{fourier coefficient}, we have the Fourier series of the functionals.   We prove that  functionals can be expanded into such Fourier series in the following theorem.
			
			\begin{thm}\label{limit}
				For  a functional $f$ that is defined on $L_\text{bound}(\Omega)$ and satisfies Assumption \ref{assump} and following two assumptions:
				
				\textbf{(i)} (Smoothness ):For any $j\in\sN_+$, $f_j\left(\sum_{i\in\sD_j} b_{i}\Phi_{i}\right)$ is a $C^1$-function respect to $b_i$ for $i\in\sD_j$ in $\left(-\frac{1}{2},\frac{1}{2}\right)^{|\sD_j|}$.
				
				\textbf{(ii)} (Existence of $a_{\vk}(f)$): For each $\vk\in\sK$, $a_{\vk}(f)$ exists
				and the following condition holds
				\begin{equation}
					\sum_{\vk \in \sK}\sum_{j\in\sA_{\vk}}\left|\int_{\left(-\frac{1}{2},\frac{1}{2}\right)^{|\sD_j|}}f_j\left(\sum_{i\in\sD_j} b_{i}\Phi_{i}\right)\left[\prod_{i\in\sD_j}\exp \left(-2\pi \I  k_ib_i\right)\,\D b_i\right]\right|<\infty,\label{ab sum}
				\end{equation}
				we have
				\begin{equation}\label{eqn:f-expansion}
					f(v)=\sum_{\vk\in\sK}a_{\vk}(f)e_{\vk}(v),~v\in L_\text{bound}(\Omega).
				\end{equation}
				The series in Eq.~(\ref{eqn:f-expansion}) is unconditionally convergent.			We call this expansion the Fourier series of the functional $f$.
			\end{thm}
		
		\begin{proof}[Proof of Theorem \ref{limit}]
			For a functional $f$ satisfying Assumption \ref{assump}, it can be written as
			\begin{align}
				f\left(\sum_{i=1}^\infty b_i\Phi_i\right)=\sum_{j=1}^\infty f_j\left(\sum_{i\in\sD_j} b_{i}\Phi_{i}\right),
			\end{align}
			where $|\sD_j|<\infty$. Furthermore, due to assumption \textbf{(i)} \ref{assump}, each $f_j$ is a $C^1$-function with respect to $b_i\in\sD_j$. Therefore, the Fourier series of $f_j$ converge to $f_j$, i.e.,
			\begin{equation}
				f_j\left(\sum_{i\in\sD_j} b_{i}\Phi_{i}\right)=\sum_{\vk\in\sK_{\sD_j}}(\widehat{f_j})_{\vk}\prod_{i\in\sD_j}\exp \left(2\pi \I  k_ib_i\right)=\sum_{\vk\in\sK_{\sD_j}}(\widehat{f_j})_{\vk}\prod_{i=1}^\infty\exp \left(2\pi \I  k_ib_i\right),
			\end{equation}
			where $\vk:=(k_1,k_2,\cdots,k_i,\cdots)$ and $(\widehat{f_j})_{\vk}$ is the Fourier coefficient of $f_j$ for the basis $\prod_{i\in\sD_j}\exp \left(2\pi \I  k_ib_i\right)$. By  the definition of $a_{\vk}(f)$ and direct calculation, we have
			\begin{equation}	a_{\vk}(f)=\sum_{j\in\sA_{\vk}}(\widehat{f_j})_{\vk}=\sum_{j\in\sA_{\vk}}\int_{\left(-\frac{1}{2},\frac{1}{2}\right)^{|\sD_j|}}f_j\left(\sum_{i\in\sD_j} b_{i}\Phi_{i}\right)\left[\prod_{i\in\sD_j}\exp \left(-2\pi \I  k_ib_i\right)\,\D b_i\right].
			\end{equation}
			Therefore
			\begin{align}
				f\left(\sum_{i=1}^\infty b_i\Phi_i\right)&=\sum_{j=1}^\infty f_j\left(\sum_{i\in\sD_j} b_{i}\Phi_{i}\right)\notag\\&=\sum_{j=1}^\infty \sum_{\vk\in\sK_{\sD_j}}(\widehat{f_j})_{\vk}\prod_{i=1}^\infty\exp \left(2\pi \I  k_ib_i\right)\notag\\&=\sum_{\vk\in\sK} \sum_{j\in\sA_{\vk}}(\widehat{f_j})_{\vk}\prod_{i=1}^\infty\exp \left(2\pi \I  k_ib_i\right)\notag\\&=\sum_{\vk\in\sK}a_{\vk}(f)\prod_{i=1}^\infty\exp \left(2\pi \I  k_ib_i\right).\label{change}
			\end{align}
			The third equation in \eqref{change} is obtained by change of order of the summations, which is due to assumption \textbf{(ii)}.
		\end{proof}
			We will give examples of the Fourier expansions of functionals at the end of the section.

		Now we define the Barron spectral space and a Hilbert space of functionals based on the Fourier series of functionals:
		
		\begin{defi}\label{space}			
			\textbf{(i)} The Barron spectral space of the continuous functionals on $L_\text{bound}(\Omega)$ that satisfy Assumption \ref{assump}  is
			\begin{equation}
				\fB_{s}[L_\text{bound}(\Omega)]:=\left\{f:f(v)=\sum_{\vk\in \sK}a_{\vk}(f)e_{\vk}(v), \ \sum_{\vk \in \sK}(1+(2\pi)^s|\vk|^s_1)|a_{\vk}(f)|<\infty\right\},\label{barron}
			\end{equation} for $s\ge 0$ with the norm
			\begin{equation}
				\|f\|_{\fB_{s}}:=\sum_{\vk \in \sK}(1+(2\pi)^s|\vk|^s_1)|a_{\vk}(f)|.
			\end{equation}
			
			\textbf{(ii)} A Hilbert space of functionals on $L_\text{bound}(\Omega)$ that satisfy Assumption \ref{assump} is
			\begin{equation}
				\fH_{s}[L_\text{bound}(\Omega)]:=\left\{f:f(v)=\sum_{\vk\in \sK}a_{\vk}(f)e_{\vk}(v), \sum_{\vk \in \sK}(1+(2\pi)^{2s}|\vk|^{2s}_1)|a_{\vk}(f)|^2<\infty\right\}\label{Hilbert}
			\end{equation}
			for $s\ge 0$ with the inner product \begin{equation}
				\left \langle f,g \right \rangle_{\fH_{s}}=\sum_{\vk \in \sK}(1+(2\pi)^{2s}|\vk|^{2s}_1)\overline{a_{\vk}(f)}\cdot a_{\vk}(g)
			\end{equation} where $\overline{a_{\vk}(f)}$ is the complex conjugate of ${a_{\vk}(f)}$.
		\end{defi}
		In this paper, we focus on the  Barron spectral space  $\fB_2[L_\text{bound}(\Omega)]$ and the Hilbert space $\fH_{1}[L_\text{bound}(\Omega)]$, and for simplicity of notations,  denote them as $\fB[L_\text{bound}(\Omega)]$  and $\fH[L_\text{bound}(\Omega)]$, respectively. Note that it is easy to check that $\fH_{s}[L_\text{bound}(\Omega)]$ defined above is a Hilbert space.

		Many widely used functionals in science and engineering satisfy  Assumption \ref{assump} and the defined Fourier expansion applies. Here we give some examples.
		\begin{exam}\label{exam1}
			For a linear functional $f$, $f\left(\sum_{i=1}^\infty b_i\Phi_i\right)=\sum_{i=1}^\infty f\left( b_{i}\Phi_{i}\right)$. In this case, Assumption \ref{assump} is satisfied with
			$\sD=\{\sD_j\}_{j=1}^{\infty}=\{\{1\},\{2\},\cdots,\{j\},\cdots\}$. For any $\vk\in \sK_{\sD_j}$, there is only one component of $\vk$, i.e., $k_j$, is nonzero, and for any $v=\sum_{i=1}^\infty b_i\Phi_i$, we have
			\begin{flalign*}
				e_{\vk}\left(v\right)=&\exp \left(2\pi \I  k_jb_j\right), \ \ \ a_{\vzero}(f):=\sum_{i=1}^\infty\int_{\left(-\frac{1}{2},\frac{1}{2}\right)}f\left( b_{i}\Phi_{i}\right)\,\D b_i,
				\notag\\
				a_{\vk}(f)=&\int_{\left(-\frac{1}{2},\frac{1}{2}\right)}f\left( b_{j}\Phi_{j}\right)\exp \left(-2\pi \I  k_jb_j\right)\,\D b_j, \ \ \vk\neq\vzero.
			\end{flalign*}
			Therefore, the Fourier expansion \eqref{eqn:f-expansion} of the linear functional $f$ is
			\begin{align}			
				f\left(\sum_{i=1}^\infty b_i\Phi_i\right)=&\sum_{i=1}^\infty\sum_{k_i\in\sZ}\left(\int_{\left(-\frac{1}{2},\frac{1}{2}\right)}f\left( b_{i}\Phi_{i}\right)\exp \left(-2\pi \I  k_ib_i\right)\,\D b_i\right)\exp \left(2\pi \I  k_ib_i\right).
			\end{align}
		The functional $f$ is in the  Barron spectral space	if it further satisfies the condition in Definition~\ref{space}(i).

		Note that linear functionals play important roles in both theoretical studies and applications in science and engineering, e.g., integration, inner product with a given function, solution of initial/boundary value problem of linear PDEs at some given points, etc. This important class of functionals always satisfy Assumption \ref{assump}.
	\end{exam}
	
	\begin{exam}\label{exam2}	
		Consider the energy functional $E(v)=\int_{(0,1)^d}\frac{1}{2}\alpha\left|\nabla v\right|^2\,\D x$ where $\alpha>0$ and $d\in\sN$. Here $v\in\Omega=H^1(0,1)^d$ with $\frac{\partial v}{\partial \nu}=0$ for $x\in\partial [0,1]^d$ (where $\frac{\partial }{\partial \nu}$ is the outer normal derivative on the boundary), and an orthogonal basis of $\Omega$ is 
		$\{\Phi_p(x)=\prod_{j=1}^d\sqrt{2}\cos(\pi p_j x_j)\}_{p\in\sN^d\backslash\{0\}}\cup\{\Phi_0=1\}$,
		where $p_j$ and $x_j$ are components of $p$ and $x$, respectively.
		For any $v=\sum_{p\in\sN^d} b_p\Phi_p$, the energy can be written as
		$$E\left(v\right)=\int_{(0,1)^d}\frac{1}{2}\alpha\left|\sum_{p\in\sN^d} \big(\pi b_p\Phi_p)p\right|^2\D x
		=\sum_{p\in\sN^d} \frac{1}{2}\pi^2\alpha |p|^2b_p^2
		=\sum_{p\in\sN^d} E\big(b_p\Phi_p\big).$$
		Here $E\big(b_p\Phi_p\big)=\frac{1}{2}\pi^2 \alpha |p|^2 b_p^2$.
		In this case, we also have that Assumption \ref{assump} is satisfied with $\sD=\{\sD_p\}_{p\in\sN^d}$, where $\sD_p=\{p\}$ for $p\in\sN^d$.
		We further restrict the domain to
		$$\left\{v(x)=\sum_{p\in\sN^d} b_p\Phi_p: -B_p< b_p<B_p,\ B_p=C|p|^{-d/2-1-\varepsilon}, \  p\in\sN^d\right\}$$
		where $C,\varepsilon>0$, so that the energy is well-defined.
		For any $\vk\in \sK_{\sD_p}$, there is only one component of $\vk$, denoted by $k_p$, is nonzero,    and we have
		\begin{flalign*}
			&e_{\vk}(v)=\frac{1}{(2B_p)^{\frac{1}{2}}}\exp\left(\frac{\pi}{B_p}\I k_p b_p\right), \ \ a_{\vzero}(E)=\sum_{p\in\sZ^d}\frac{1}{(2B_p)^{\frac{1}{2}}}\int_{-B_p}^{B_p}E\big(b_p\Phi_p\big)\,\D b_p,
			\notag\\ 				
			&a_{\vk}(E)=\frac{1}{(2B_p)^{\frac{1}{2}}}\int_{-B_n}^{B_n}E\big(b_p\Phi_p\big)\exp\left(-\frac{\pi}{B_p}\I k_p b_p\right)\,\D b_p.\notag\\
		\end{flalign*}
		The Fourier expansion \eqref{eqn:f-expansion}  of this energy functional $E$ is
		\begin{align}			
			E=&\sum_{p\in\sN^d}\sum_{k_p\in\sZ}\left(\frac{1}{2B_p}\int_{-B_p}^{B_p}E\big(b_p\Phi_p\big)\exp\left(-\frac{\pi}{B_p}\I k_p b_p\right)\,\D b_p\right)\exp\left(\frac{\pi}{B_p}\I k_p b_p\right).
		\end{align}
The functional $E$ is in the  Barron spectral space	if it further satisfies the condition in Definition~\ref{space}(i).

Note that our method also applies to many other similar nonlinear functionals such as the elastic energy of a deformed materials $E=\frac{1}{2}\sum_{ijkl}C_{ijkl}\varepsilon_{ij}\varepsilon_{kl}$, where $\{C_{ijkl}\}$ is the elastic constant tensor, $\varepsilon_{ij}=\frac{1}{2}(\frac{\partial u_j}{\partial x_i}+\frac{\partial u_i}{\partial x_j})$ is the strain tensor,  and $u_i$ is a component of the displacement vector.
\end{exam}

We further consider an example in which the functions are defined in a domain of discrete points. In Physics, chemistry, and materials science, atomistic models are commonly used as a tool to study the materials properties~\cite{MD}. Inter-atomic potentials are used in atomistic models, and many of them are pairwise potentials with finite interaction range. The total energy of an atomistic system depends on the positions of atoms. The number of atomics is typically very large in the atomistic simulations.
\begin{exam}\label{exam}
Consider an example of the total energy of a one dimensional atomistic system:\\
$E=\sum_{i\in\sN}\big[V(a+u_{i+1}-u_i)-V(a)\big]$,
where $V(r)$ is a two-body potential such as the Lennard-Jones potential ($V(r)=4\varepsilon\left[(\sigma/r)^{12}-(\sigma/r)^6\right]$ with $\varepsilon$ and $\sigma$ being two parameters), where $r$ is the distance between two atoms,  $a$ is a the lattice constant, and $u_i$ the displacement of the $i$-th atom. Since the number of atoms is very large, we simply write the number as infinity. Note that the convergence of the summation is not a problem in a real system because the number of atoms is always finite; or we can add some decaying condition on $\{u_i\}$ as in Example~\ref{exam2}.	

In this case,  Assumption \ref{assump} is satisfied with $\sD=\{\{1,2\},\{2,3\},\cdots,\{n,n+1\},\cdots\}$. Denote
$u=(u_1,u_2,\cdots,u_n,\cdots)$.
The Fourier expansion \eqref{eqn:f-expansion} of this energy functional $E$  is $\sum_{\vk\in\sK}a_{\vk}(E)e_{\vk}(u)$, which is
\begin{flalign}
E=&\sum_{i\in\sN}\sum_{(k_i,k_{i+1})\in\sZ^2}\exp \left(2\pi \I  k_iu_i+2\pi\I k_{i+1}u_{i+1}\right)\notag\\&\cdot \int_{\left(-\frac{1}{2},\frac{1}{2}\right)^{2}}\big[V(a+u_{j+1}-u_j)-V(a)\big]\exp(2\pi \I  k_iu_i+2\pi \I  k_{i+1}u_{i+1})\,\D u_i\,\D u_{i+1}.\label{eqn:atom}
\end{flalign} 

The functional $E$ is in the  Barron spectral space	if it further satisfies the condition in Definition~\ref{space}(i), where the Fourier coefficient $a_{\vk}(E)$ is the integral expression in Eq.~\eqref{eqn:atom}.
Note that we have similar Fourier expansion for the total energy of an atomic system with a general atomic interaction of finite range in any dimension.
\end{exam}

\begin{exam}
	Consider the following infinite dimensional nonlinear functional $f^{(1)}$ satisfying Assumption \ref{assump}:
	
	\begin{equation}
		f^{(1)}\left(\sum_{i=1}^\infty b_i\Phi_i\right)=\sum_{i=1}^\infty s_ib_i^3,
	\end{equation}
	where $s_i\in\sR$. Here $f^{(1)}$ is  well-defined in $L_\text{bound}(\Omega)$ when $\sum_{i=1}^\infty{|s_i|}<\infty$. In this case, $\sD=\{\{1\},\{2\},\cdots,\{n\},\cdots\}$.
	Hence we obtain that\begin{align}
		\sum_{\vk\in \sK}a_{\vk}(f^{(1)})e_{\vk}(v)=\sum_{i=1}^\infty\left(\sum_{\vk\in\sK_i}a_{\vk}(f^{(1)})e_{\vk}(v)\right)=\sum_{i=1}^\infty\left(s_i\sum_{k\in\sZ}d_k\exp(2\pi \I  kb_i)\right)
	\end{align}
	where \begin{align}
		\sK_i&:=\{\vk\in\sK:k_j=0 \ {\rm for}\ j\not=i, \ {\rm and}\ k_j\not=0 \ {\rm for} j=i\}\notag\\d_k&:=\int_{\left(-\frac{1}{2},\frac{1}{2}\right)}b_1^3\exp \left(-2\pi \I  kb_1\right)\,\D b_1.\notag
	\end{align}
	
	When $b_i\in\big(-\frac{1}{2},\frac{1}{2}\big)$, $\sum_{k\in\sZ}d_k\exp(2\pi \I  kb_i)$ is the Fourier series of $b^3_i$. Therefore, when $\sum_{i=1}^\infty{|s_i|}<\infty$, \begin{equation}
		\sum_{\vk\in \sK}a_{\vk}(f^{(1)})e_{\vk}(v)=\sum_{i=1}^\infty\left(s_i\sum_{k\in\sZ}d_k\exp(2\pi \I  kb_i)\right)=\sum_{i=1}^\infty\left(s_ib_i^3\right)=f^{(1)}\left(\sum_{i=1}^\infty b_i\Phi_i\right).
	\end{equation}
	Furthermore, $f^{(1)}\in\fB[L_\text{bound}(\Omega)]$ $\left(f^{(1)}\in\fH[L_\text{bound}(\Omega)]\right)$ when $\sum_{k\in\sZ}(1+(2\pi)^2k^{2})|d_k|<\infty$ $\left(\sum_{k\in\sN_+}(1+(2\pi)^2k^{2})|d_k|^2<\infty, \ \sum_{i\in\sN_+}|s_i|^2<\infty\right)$.
\end{exam}
\begin{exam}\label{exam6}
	Consider another example:
	\begin{equation}
		f^{(2)}\left(\sum_{i=1}^\infty b_i\Phi_i\right)=\sum_{i=1}^\infty s_ib_ib_{i+1}.
	\end{equation}Here $f^{(2)}$ is  well-defined in $L_\text{bound}(\Omega)$ when $\sum_{i=1}^\infty{|s_i|}<\infty$. In this case, \[\sD=\{\{1,2\},\{2,3\},\cdots,\{n,n+1\},\cdots\},\] and
	\begin{align}
		a_{\vk}(f^{(2)})=\int_{\left(-\frac{1}{2},\frac{1}{2}\right)^{2}}s_ib_ib_{i+1}\exp(2\pi \I  k_ib_i)\,\D b_i\,\D b_{i+1}+\int_{\left(-\frac{1}{2},\frac{1}{2}\right)^{2}}s_{i-1}b_ib_{i-1}\exp(2\pi \I  k_ib_i)\,\D b_i\,\D b_{i-1},
	\end{align} when $\vk\in\sK_i:=\{\vk\in\sK:k_j=0 \ {\rm for} \ j\not=i, \ {\rm and}\ k_j\not=0 \ {\rm for}\ j=i\}$, and $s_0=b_0=0$. It can be calculated that \begin{align}
		a_{\vk}(f^{(2)})=\int_{\left(-\frac{1}{2},\frac{1}{2}\right)^{2}}s_ib_ib_{i+1}\exp(2\pi \I  k_ib_i+2\pi \I  k_{i+1}b_{i+1})\,\D b_i\,\D b_{i+1}
	\end{align} when $\vk\in\bar{\sK}_i:=\{\vk\in\sK:k_j=0\ {\rm for}\ j\not=i,i+1, \ {\rm and}\ k_j\not=0 \ {\rm for}\ j=i,i+1\}$. Otherwise $a_{\vk}(f^{(2)})=0$. Therefore,\begin{equation}
		\sum_{\vk\in \sK}a_{\vk}(f^{(2)})e_{\vk}(v)=\sum_{i=1}^\infty\left(\sum_{\vk\in\sK_i\cup\sK_{i+1}\cup\bar{\sK}_{i}}a_{\vk}(f^{(2)})\exp(2\pi \I k_ib_i+2\pi \I  k_{i+1}b_{i+1})\right)=\sum_{i=1}^\infty\left(s_ib_ib_{i+1}\right).
	\end{equation}
\end{exam}
In the above two examples, there is a sequence $\{s_i\}_{i=1}^\infty$ to make the functionals well-defined in $L_\text{bound}(\Omega)$ (i.e., the series converges).
Note that alternatively, convergence of the series can also be achieved by restricting the domain of functionals. For example, if we want to approximate functionals such as the $L_2$-norm: $f(v)=\|v\|_{L_2(0,1)}$, we can restrict the domain of the functionals to $\left\{v=\sum_{i=1}^\infty b_i\Phi_i: -Ci^{-1/2-\varepsilon}< b_i< Ci^{-1/2-\varepsilon}, i\in\sN_+\right\}$ where $C,\varepsilon>0$, and redefine $a_{\vk}(f)$ as that in Eq.~\eqref{revisit1}.

\begin{exam}\label{sin}
	Consider the functional $f(v)=\int_0^1 v^3\,\D x $ and an orthogonal basis of functions on $(0,1)$:
	$\{\Phi_j(x)=\sqrt{2}\cos(\pi j x)\}_{j\in\sN\backslash\{0\}}\cup\{\Phi_0=1\}$. For any $v=\sum_{j\in\sN} b_j\Phi_j$, the functional $f$ can be written as\begin{align}
		f\left(\sum_{j\in\sN} b_j\Phi_j\right)=&\sum_{j=0}^\infty\int_0^1 b_j^3\Phi_j^3(x)\,\D x+3\sum_{j_1\not=j_2}\int_0^1 b_{j_1}^2b_{j_2}\Phi_{j_1}^2(x)\Phi_{j_2}(x)\,\D x\notag\\&+6\sum_{\text{$j_1,j_2,j_3$ different}}\int_0^1 b_{j_1}b_{j_2}b_{j_3}\Phi_{j_1}(x)\Phi_{j_2}(x)\Phi_{j_3}(x)\,\D x.\label{form1}
	\end{align}

%
%

Using the definition of $\{\Phi_j(x) \}$, it can be calculated that
\begin{align}
	f\left(\sum_{j\in\sN} b_j\Phi_j\right)=b_0^3+3\sum_{j_1=1}^\infty b_{j_1}^2b_{0}+\frac{3\sqrt{2}}{2}\sum_{2j_1=j_2\atop j_1,j_2>0} b_{j_1}^2b_{j_2}+3\sqrt{2}\sum_{\text{$j_1,j_2,j_3$  different}\atop j_1+j_2=j_3} b_{j_1}b_{j_2}b_{j_3}.
\end{align}

Note that the convergence of the summation is not a problem since we can add some decaying condition on $\{b_j\}$ as in Example~\ref{exam2}:\[\left\{v(x)=\sum_{j\in\sN} b_j\Phi_j: -C_p< b_p<C_p,\ C_p=C|j|^{-1-\varepsilon}, \  j\in\sN\right\}\]	where $C,\varepsilon>0$. In this case, we also have that Assumption \ref{assump} is satisfied with $\sD=\{\{0\},\{0,j_1\}_{j_1=1}^\infty,\{j_1,j_2\}_{2j_1=j_2\atop j_1,j_2>0},\{j_1,j_2,j_3\}_{\text{$j_1,j_2,j_3$ different},\atop j_1+j_2=j_3}\}$. We also can write down the Fourier expansion of $f(v)$. We put the detail in the end of this example.

Now we present the Fourier coefficients of this functional $f(v)$ in the Fourier expansion \eqref{eqn:f-expansion}. 
For any $\vk\in \sK$ with only one nonzero component of $\vk$, denoted by $k_{s_1}$, and we have when $s_1=0$, \begin{align}
	e_{\vk}(v)=&\frac{1}{(2C_0)^{\frac{1}{2}}}\exp\left(\frac{\pi}{C_0}\I k_0 b_0\right),\notag\\a_{\vk}(f)=&\frac{1}{(2C_0)^{\frac{1}{2}}}\int_{-C_0}^{C_0} b_0^3e_{\vk}\,\D b_0+3\sum_{j_1=1}^\infty\frac{1}{(2C_0)^{\frac{1}{2}}2C_{j_1}}\int_{-C_{j_1}}^{C_{j_1}}\int_{-C_0}^{C_0} b_0b_{j_1}^2e_{\vk}\,\D b_0\,\D b_{j_1};
\end{align} for an odd integer $s_1$, \begin{align}
	e_{\vk}(v)=&\frac{1}{(2C_{s_1})^{\frac{1}{2}}}\exp\left(\frac{\pi}{C_{s_1}}\I k_{s_1} b_{s_1}\right),\notag\\a_{\vk}(f)=&\frac{3}{(2C_{s_1})^{\frac{1}{2}}2C_0}\int_{-C_{0}}^{C_{0}}\int_{-C_{s_1}}^{C_{s_1}} b_0b_{s_1}^2e_{\vk}\,\D b_{s_1}\,\D b_0+\frac{3\sqrt{2}}{2}\frac{1}{(2C_{s_1})^{\frac{1}{2}}2C_{2s_1}}\int_{-C_{s_1}}^{C_{s_1}}\int_{-C_{2s_1}}^{C_{2s_1}} b_{2s_1}b_{s_1}^2e_{\vk}\,\D b_{2s_1}\,\D b_{s_1}\notag\\&+3\sqrt{2}\sum_{\text{$s_1,j_1,j_2$ different}\atop s_1+j_1=j_2\text{ or }j_1+j_2=s_1}\frac{1}{(2C_{s_1})^{\frac{1}{2}}4C_{j_1}C_{j_2}}\int_{-C_{j_2}}^{C_{j_2}}\int_{-C_{j_1}}^{C_{j_1}}\int_{-C_{s_1}}^{C_{s_1}} b_{j_2}b_{j_1}b_{s_1}e_{\vk}\,\D b_{s_1}\,\D b_{j_1}\,\D b_{j_2};
\end{align} and for an even integer $s_1\not=0$, \begin{align}
	e_{\vk}(v)=&\frac{1}{(2C_{s_1})^{\frac{1}{2}}}\exp\left(\frac{\pi}{C_{s_1}}\I k_{s_1} b_{s_1}\right),\notag\\a_{\vk}(f)=&\frac{3}{(2C_{s_1})^{\frac{1}{2}}2C_0}\int_{-C_{0}}^{C_{0}}\int_{-C_{s_1}}^{C_{s_1}} b_0b_{s_1}^2e_{\vk}\,\D b_{s_1}\,\D b_0\notag\\&+\frac{3\sqrt{2}}{2}\frac{1}{(2C_{s_1})^{\frac{1}{2}}2C_{2s_1}}\int_{-C_{s_1}}^{C_{s_1}}\int_{-C_{2s_1}}^{C_{2s_1}} b_{2s_1}b_{s_1}^2e_{\vk}\,\D b_{2s_1}\,\D b_{s_1}\notag\\&+\frac{3\sqrt{2}}{2}\frac{1}{(2C_{s_1})^{\frac{1}{2}}2C_{s_1/2}}\int_{-C_{s_1}}^{C_{s_1}}\int_{-C_{s_1/2}}^{C_{s_1/2}} b_{s_1}b_{s_1/2}^2e_{\vk}\,\D b_{s_1/2}\,\D b_{s_1}\notag\\&+3\sqrt{2}\sum_{\text{$s_1,j_1,j_2$ different}\atop s_1+j_1=j_2\text{ or }j_1+j_2=s_1}\frac{1}{(2C_{s_1})^{\frac{1}{2}}4C_{j_1}C_{j_2}}\int_{-C_{j_2}}^{C_{j_2}}\int_{-C_{j_1}}^{C_{j_1}}\int_{-C_{s_1}}^{C_{s_1}} b_{j_2}b_{j_1}b_{s_1}e_{\vk}\,\D b_{s_1}\,\D b_{j_1}\,\D b_{j_2}.
\end{align}

For any $\vk\in \sK$ with only two nonzero components, denoted by $k_{s_1},k_{s_2}$, when $s_2=0$, \begin{align}
	e_{\vk}(v)=&\frac{1}{(4C_0C_{s_1})^{\frac{1}{2}}}\exp\left(\frac{\pi}{C_{0}}\I k_{0} b_{0}+\frac{\pi}{C_{s_1}}\I k_{s_1} b_{s_1}\right),\notag\\a_{\vk}(f)=&3\frac{1}{(4C_0C_{s_1})^{\frac{1}{2}}}\int_{-C_{s_1}}^{C_{s_1}}\int_{-C_0}^{C_0} b_0b_{s_1}^2e_{\vk}\,\D b_0\,\D b_{s_1};
\end{align} when $s_2=2s_1$, \begin{align}
	e_{\vk}(v)=&\frac{1}{(4C_{s_2}C_{s_1})^{\frac{1}{2}}}\exp\left(\frac{\pi}{C_{s_2}}\I k_{s_2} b_{s_2}+\frac{\pi}{C_{s_1}}\I k_{s_1} b_{s_1}\right),\notag\\a_{\vk}(f)=&3\frac{1}{(4C_0C_{s_1})^{\frac{1}{2}}}\int_{-C_{s_1}}^{C_{s_1}}\int_{-C_{s_2}}^{C_{s_2}} b_{s_2}b_{s_1}^2e_{\vk}\,\D b_{s_2}\,\D b_{s_1}\notag\\&+3\sqrt{2}\frac{1}{(4C_{s_1}C_{s_2})^{\frac{1}{2}}2C_{s_1+s_2}}\int_{-C_{s_1+s_2}}^{C_{s_1+s_2}}\int_{-C_{s_2}}^{C_{s_2}}\int_{-C_{s_1}}^{C_{s_1}} b_{j_2}b_{s_2}b_{s_1}e_{\vk}\,\D b_{s_1}\,\D b_{s_2}\,\D b_{s_1+s_2};
\end{align} and when $s_2\not=2s_1$, $s_1\not=2s_2$ and $s_2\not=s_1$, \begin{align}
	e_{\vk}(v)=&\frac{1}{(4C_{s_2}C_{s_1})^{\frac{1}{2}}}\exp\left(\frac{\pi}{C_{s_2}}\I k_{s_2} b_{s_2}+\frac{\pi}{C_{s_1}}\I k_{s_1} b_{s_1}\right),\notag\\a_{\vk}(f)=&3\sqrt{2}\frac{1}{(4C_{s_1}C_{s_2})^{\frac{1}{2}}2C_{s_1+s_2}}\int_{-C_{s_1+s_2}}^{C_{s_1+s_2}}\int_{-C_{s_2}}^{C_{s_2}}\int_{-C_{s_1}}^{C_{s_1}} b_{j_2}b_{s_2}b_{s_1}e_{\vk}\,\D b_{s_1}\,\D b_{s_2}\,\D b_{s_1+s_2}.
\end{align}

For any $\vk\in \sK$ with only three nonzero components, denoted by $k_{s_1},k_{s_2}, k_{s_3}$, when $s_1+s_2=s_3$, \begin{align}
	e_{\vk}(v)=&\frac{1}{(8C_{s_3}C_{s_2}C_{s_1})^{\frac{1}{2}}}\exp\left(\frac{\pi}{C_{s_3}}\I k_{s_3} b_{s_3}+\frac{\pi}{C_{s_2}}\I k_{s_2} b_{s_2}+\frac{\pi}{C_{s_1}}\I k_{s_1} b_{s_1}\right),\notag\\a_{\vk}(f)=&3\sqrt{2}\frac{1}{(8C_{s_1}C_{s_2}C_{s_3})^{\frac{1}{2}}}\int_{-C_{s_3}}^{C_{s_3}}\int_{-C_{s_2}}^{C_{s_2}}\int_{-C_{s_1}}^{C_{s_1}} b_{s_3}b_{s_2}b_{s_1}e_{\vk}\,\D b_{s_1}\,\D b_{s_2}\,\D b_{s_3}.
\end{align}
\end{exam}

Since our method works similarly for the functional $f(v)=\int_0^1 v^q\,\D x $, $q\in\sZ$, it can be applied to more general cases, e.g., $g(v)=\int_0^1 \sin v\,\D x $ by using Taylor expansions:\begin{align}
	\int_0^1 \sin v\,\D x=\sum_{n=0}^{\infty}\int_0^1  \frac{(-1)^n v(x)^{2 n+1}}{(2 n+1) !}\,\D x,~|v|<1.
\end{align}
For this case, Assumption \ref{assump} is also satisfied.
This shows that our method can be applied to more general functionals by using Taylor expansions together with Assumption~\ref{assump}.

\section{Approximation of Functionals Based on Barron Spectral Space}\label{sec:sec3}
Based on the Barron spectral space for functionals defined above, we prove the  approximation of functionals by neural networks without curse of dimensionality.  
This is the main result of this paper and is summarized in the following theorem, whose proof  is given in Appendix~\ref{Asec:proof-main}.
\begin{thm}\label{main}
For any functional $f\in \fB[L_\text{bound}(\Omega)]$, there is an $f_m\in\fG_\text{ReLU,m,f}$, where
\begin{flalign}
\fG_\text{ReLU,m,f}:=&\Bigg\{g\left(v\right)=c+\sum_{j=1}^m\gamma_j \text{ReLU} \left(\sum_{i\in \sS_{\vk_j}}w_{ij}b_i-t_j\right): v=\sum_{i=1}^\infty b_i\Phi_i,\vk_j\in\sK_f \notag\\&
|c|\leq 2\|f\|_{\fB},\sum_{j=1}^m|\gamma_j| \leq 4\|f\|_{\fB}, |\vw_j|_1=1,|t|\le 1\Bigg\},
\end{flalign}
with $\sS_{\vk_j}:=\{i\in\sN: k_{ij}\not=0, \vk_j=(k_{1j},k_{2j},\cdots,k_{ij},\cdots)\in\sK_f\}$ and $\sK_f:=\{\vk\in\sK: a_{\vk}(f)\not=0\}$,
that satisfies
\begin{equation}
\|f-f_m\|_{\fH}\le\frac{4\sqrt{5}\|f\|_{\fB}}{\sqrt{m}}.
\end{equation}
\end{thm}

Note that  there is no dependence on the dimensionality in the error estimate in this theorem, which overcomes the curse of the dimensionality. 
In fact,			Theorem \ref{main} states that the functional $f$ in $\fB[L_\text{bound}(\Omega)]$ can be approximated by a two layer network with $m$ nodes in the hidden layer with $O(1/\sqrt{m})$ accuracy in $\fH[L_\text{bound}(\Omega)]$. For the $j$-th node in the hidden layer, the information of $f$ associated with its $j$-th chosen Fourier basis functional is used,
whose dimension is $|\sS_{\vk_j}|$, i.e., the number of nonzero components of $\vk_j$, and is bounded by the number of  elements in the block where the Fourier basis functional lies in by Assumption \ref{assump}.
As a result, only information of finite dimensions of $f$ is needed in this $O(1/\sqrt{m})$ approximation. Note that $|\sS_{\vk_j}|$ depends only on $f$. Theorem \ref{main} in fact provides a way to approximate a functional in an infinite dimensional space by a neural network with finite number of parameters, i.e., $O(mN_{f,m})$, where $N_{f,m}=\max_{1\leq j\leq m} |\sS_{\vk_j}|$. This approximation does not have curse of dimensionality.

In the DeepONet method as reviewed in the introduction section, the domain of the input functions is first discretized to reduce the infinite dimensional problem into a finite dimensional one and then the functional is learned by a neural network,  and such process is suffered from curse of dimensionality (cf. the last two errors in Eq.~\eqref{error}). In contrast, our Theorem \ref{main} approximates the functional directly by a neural network without discretization of the domain of the input functions, and there is no curse of dimensionality. Moreover, since our method does not approximate the domain of the input functions, the number of parameters and the network structure in our method only depends on the functional,  thus it is not sensitive to the input functions in training.

%


In practice, when we set up the neural network, we do not know $\sS_{\vk_j}$, $j=1,2,\cdots,m$, because we do not know $f$ yet. 
This can be solved by choosing a large $N$ and defining
\begin{equation}f_m\left(\sum_{i=1}^\infty b_i\Phi_i(\vx)\right)=c+\sum_{j=1}^m\gamma_j \text{ReLU} \left(\sum_{i=1}^{N}w_{ij}b_i-t_j\right).\label{enough}\end{equation}
When $N\ge N_{f,m}$, this form of $f_m$ includes all the functionals in $\fG_{\text{ReLU},m,f}$ including the desired one with error of $O(1/\sqrt{m})$ given in Theorems \ref{main}, thus it
is able to approximate $f$ with $O(1/\sqrt{m})$  accuracy. 
In this case, the number of parameters in the network is  $O(mN)$, instead of exponential dependence on the dimension $N$; and once $N\geq N_{f,m}$, the desired approximation of $O(1/\sqrt{m})$ accuracy given in Theorems \ref{main} is included by $f_m$, and no further increase of $N$ is needed. Note that $N_{f,m}$ and accordingly $N$ depend only on the functional $f$ and do not depend on the input functions. There is still no curse of dimensionality by using this treatment in practice. The input data for the training of the network will be the finite Fourier coefficients $\{b_i\}_{i=1}^N$ of the input functions.

When the functional $f$  has the special structure with
$\sD=\{\{1\},\{2\},\cdots,\{n\},\cdots\}$,
such as linear functionals and the gradient energy functionals as given in Examples \ref{exam1} and \ref{exam2}, the approximating neural network in Theorem \ref{main} can be made more efficient and accurate; see the discussion in Section~\ref{sec:app-D1}.

\section{Reducing the Number of Parameters in $\fG_\text{ReLU,m,f}$}

Theorem \ref{main} shows that a functional $f\in \fB[L_\text{bound}(\Omega)]$ can be approximated well from the set $\fG_\text{ReLU,m,f}$  without curse of dimensionality. A challenge when using this theorem in deep learning is that the set $S_{\vk_j}$ for each $j$ associated with $f$ that appear in $\fG_\text{ReLU,m,f}$ are unknown because $f$ is a functional to be learnt.

Recall that when the Fourier basis is used, a smoother function has faster decaying Fourier coefficients. Here we consider a generalization of this property for any basis $\{\Phi_{i}\}$, so that the number of parameters for $i\in \sS_{\vk_j}$ can be replaced by a fixed number $N$ based on the domain of the functional. 

We restrict the domain of the functional to
\begin{equation}
L_\text{cut}(\Omega):=\left\{v(\vx)\in L_\text{bound}(\Omega): v(\vx)=\sum_{i=1}^\infty b_i\Phi_i(\vx),~\vx\in K_1, ~|b_i|<\delta\ {\rm for} \ i>N\right\},
\end{equation}
where $\delta$ is a small constant and $N\in\sN_+$.
Based on $L_\text{cut}(\Omega)$, we modify the Fourier coefficients in Definition \ref{fourier coefficient} as:
\begin{flalign}	a_{\vk}(f):=&\sum_{j\in\sA_{\vk}}\frac{1}{(2\delta)^{|\sD_j\cap[N+1,+\infty)|/2}}\int_{\left(-\delta,\delta\right)^{|\sD_j\cap[N+1,+\infty)|}\times\left(-\frac{1}{2},\frac{1}{2}\right)^{|\sD_j\cap[0,N]|}}f_j\left(\sum_{i\in\sD_j} b_{i}\Phi_{i}\right)\notag\\&\cdot\left[\prod_{i\in\sD_j\cap[0,N]}\exp \left(-2\pi \I  k_ib_i\right)\,\D b_i\prod_{i\in\sD_j\cap[N+1,+\infty)}\exp \left(-\frac{1}{\delta}\pi \I  k_ib_i\right)\,\D b_i\right].\label{revisit1}
\end{flalign}

With the space $L_\text{cut}(\Omega)$, we can use the parameter $N$ in $L_\text{cut}(\Omega)$ instead of the unknown number of parameters for $i\in \sS_{\vk_j}$ of $f$ in the approximation by neural network. The result is summarized in the following theorem.
\begin{thm}\label{cutoff1}
For any $f_m\in\fG_\text{ReLU,m,f}$ that is given by
\begin{equation}
f_m\left(\sum_{i=1}^{\infty}b_i\Phi_i\right)=c+\sum_{j=1}^m\gamma_j \text{ReLU} \left(\sum_{i\in \sS_{\vk_j}}w_{ij}b_i-t_j\right),\label{ori}
\end{equation} there are $f_m^*,f_m^{**}$ from \begin{align}
&\fG^*_\text{ReLU,m,f}:=\notag\\&\Bigg\{g\left(\sum_{i=1}^{\infty}b_i\Phi_i\right)=c+\sum_{j=1}^m\gamma_j \text{ReLU} \left(\sum_{i=1}^{N}w_{ij}b_i-t_j\right):2|c|,\sum_{j=1}^m|\gamma_j| \leq 4 \|f\|_{\fB}, |\vw_j|_1=1,|t|\le 1\Bigg\},\label{revisit}
\end{align} such that for any $v\in L_\text{cut}(\Omega)$,
\begin{align}
|f^*_m(v)-f_m(v)|&\le 4\|f\|_{\fB}\delta,\\
\|f^{**}_m-f_m\|_{\fH}&\le 2\sqrt{13}\|f\|_{\fB}(2\delta)^{1/2}.
\end{align}
\end{thm}
Here the constant $N$ in $\fG^*_\text{ReLU,m,f}$ in \eqref{revisit} is that in the definition of $L_\text{cut}(\Omega)$.

Theorem \ref{cutoff1} combined with Theorem \ref{main} show that there exists a neural network in $G^*_{ReLU,m,f}$ that approximates $f$ well  on the domain $L_\text{cut}(\Omega)$, and the error is
\begin{equation}
\|f^{**}_m-f\|_{\fH}\le\left( \frac{4\sqrt{5}}{\sqrt{m}}+2\sqrt{13}(2\delta)^{1/2}\right)\|f\|_{\fB}.
\end{equation}

For example, if we want to learn the gradient energy in Example \ref{exam2}, where the functional is $E(u)=\int_{0}^{1}\frac{1}{2}\alpha\left(\frac{\,\D u}{\,\D x}\right)^2\,\D x$, we may restrict the domain of the functional to
$\left\{v=\sum_{n=0}^\infty b_n\Phi_n: -Cn^{-3/2-\varepsilon}< b_n< Cn^{-3/2-\varepsilon}, n\in\sN\right\}$, where $C,\varepsilon>0$ and $\{\Phi_n=\sqrt{2}\cos(\pi nx)\}_{n=1}^\infty\cup \{\Phi_0=1\}$.
In this case, we can select $\delta=CN^{-3/2-\varepsilon}$ for a large $N$ in $L_\text{cut}(\Omega)$. Combining Theorem \ref{main} and $\ref{cutoff1}$,  the approximation error of $\fG^*_\text{ReLU,m,f}$ in this case is $O(1/\sqrt{m})+O(N^{-3/4-\varepsilon/2})$. For the second error, we can obtain a smaller error for a fixed $N$, when we consider a functional on a space with smoother functions, such as $H^k(0,1)$. For $H^k(0,1)$, we should consider the space
\[\left\{v=\sum_{n=0}^\infty b_n\Phi_n: -Cn^{-(1+k)/2-\varepsilon}< b_n< Cn^{-(1+k)/2-\varepsilon}, n\in\sN\right\},\]
where $C,\varepsilon>0$, and the second error is $O(N^{-(1+k)/4-\varepsilon/2})$. With the same level of the error, a space with smoother functions will lead to fewer parameters in the neural network.

\section{Proofs of Theorem \ref{main} and \ref{cutoff1}}
\subsection{Proofs of Theorem \ref{main} and related propositions and lemmas}\label{Asec:proof-main}
We first sketch the main steps in the proof of Theorem \ref{main} and then give the full proof.

{\bf Main steps of the proof of Theorem \ref{main}:}

\textbf{(i)} Show that for any $f\in\fB[L_\text{bound}(\Omega)]$, $f-a_{\vzero}(f)$ is in the closure of the convex hull of a set $\fG_{\sK_f \backslash\{\vzero\},f}$ in Hilbert space $\fH[L_\text{bound}(\Omega)]$, where
\begin{align}
	&\fG_{\sK_f \backslash\{\vzero\},f}:=\notag\\&\Bigg\{g\left(\sum_{i=1}^\infty b_i\Phi_i\right)=\frac{\gamma}{1+(2\pi)^2|\vk|_1^2} \cos \left[2\pi \left(\sum_{i\in\sS_{\vk_j}}k_ib_i+b\right)\right]:|\gamma| \leq \|f\|_{\fB}, b \in[0,1],\vk\in\sK_f\backslash\{\vzero\}\Bigg\}.
\end{align}
\textbf{(ii)} Show that each element in $\bar{c}+\fG_{\sK_f \backslash\{\vzero\},f}$, where $|\bar{c}|\le \|f\|_{\fB}$, is in the closure of the convex hull of a set $\fG_{\text{ReLU},f}$ in Hilbert space $\fH[L_\text{bound}(\Omega)]$, where
\begin{align}
	&\fG_{\text{ReLU},f}:=\notag\\&\left\{g\left(\sum_{i=1}^\infty b_i\Phi_i\right)=c+\gamma \text{ReLU} \left(\sum_{i\in\sS_{\vk_j}}w_ib_i-t\right):2|c|,|\gamma| \leq 4\|f\|_{\fB}, |\vw|_1=1,|t|\le 1,\vk\in\sK_f\backslash\{\vzero\}\right\}.\label{eqn:GReLU}
\end{align}
\textbf{(iii)} Then using Lemma \ref{convex} below, we can show that the convex combination of elements in $\fG_{\text{ReLU},f} $ can approximate $f(v)$ with the convergence rate $O(1/\sqrt{m})$, where $m$ is the number of elements in the convex combination.
\begin{lem}[\cite{barron1993universal,pisier1981remarques}]\label{convex}
	Let $h$ belongs to the closure of the convex hull of a set $\fG$ in Hilbert space. Denote the Hilbert norm as $\|~\cdot~\|$. If each element of $\fG$ be upper bounded by $B>0$. Then for every $m\in\sN$, there are $\{h_i\}_{i=1}^m\subset\fG$ and $\{c_i\}_{i=1}^m\subset[0,1]$ with $\sum_{i=1}^m c_i=1$, such that
	\begin{equation}
		\left\|h-\sum_{i=1}^m c_ih_i\right\|^2\le \frac{B^2}{m}.
	\end{equation}
\end{lem}
\begin{prop}\label{into}
	\[
	\fB[L_\text{bound}(\Omega)] \hookrightarrow \fH[L_\text{bound}(\Omega)].
	\]
\end{prop}	
\begin{proof}[Proof of Proposition \ref{into}]
	For any functional $f$ in $\fB[L_\text{bound}(\Omega)]$, we have that for any $\vk\in\sK$, $$|a_{\vk}(f)|\le \sum_{\vk \in \sK}|a_{\vk}(f)|\leq\|f\|_{\fB}.$$
	Hence we have
	\begin{align}
		\|f\|_{\fH}^2=\sum_{\vk \in \sK}(1+(2\pi)^2|\vk|^{2})|a_{\vk}(f)|^2\le\sum_{\vk \in \sK}(1+(2\pi)^2|\vk|^{2})|a_{\vk}(f)|\cdot\|f\|_{\fB}\le\|f\|_{\fB}^2.
	\end{align}
\end{proof}

Based on Proposition \ref{into}, we can prove the following lemma:
\begin{lem}\label{convex1}
	For any functional $f\in\fB[L_\text{bound}(\Omega)]$, $f-a_{\vzero}(f)$ is in the closure of the convex hull of the set $\fG_{\sK_f \backslash\{\vzero\},f}$ in the Hilbert space $\fH[L_\text{bound}(\Omega)]$.
\end{lem}	
\begin{proof}[Proof of Lemma \ref{convex1}]
	Since $f\in\fB[L_\text{bound}(\Omega)]$, we have
	$$f(v)=\sum_{\vk\in \sK}a_{\vk}(f)e_{\vk}(v)=\sum_{\vk\in \sK_f}a_{\vk}(f)e_{\vk}(v),$$
	for any $v\in  L_\text{bound}(\Omega)$.
	Denote  $v=\sum_{i=1}^{\infty}b_i\Phi_i$, and $\sS_{\vk_j}:=\{i:\vk_j:=(k_{1j},k_{2j},\cdots),k_{ij}\not=0,\vk\in\sK_f\}$. We obtain
	
	\begin{align}
		&f(v)-a_{\vzero}(f)\notag\\=&\sum_{\vk\in \sK_f \backslash\{\vzero\}}a_{\vk}(f)e_{\vk}(v)\notag\\
		=&\sum_{\vk\in\sK_f \backslash\{\vzero\}}a_{\vk}(f)\prod_{i=1}^\infty\exp \left(2\pi \I  k_ib_i\right)\notag\\
		=&\sum_{\vk\in \sK_f \backslash\{\vzero\}}a_{\vk}(f)\prod_{i\in\sS_{\vk_j}}\exp \left(2\pi \I  k_ib_i\right)\notag\\
		=&\sum_{\vk\in \sK_f \backslash\{\vzero\}}a_{\vk}(f)\exp \left(2\pi \I  \sum_{i\in\sS_{\vk_j}}k_ib_i\right)\notag\\
		=&\sum_{\vk\in \sK_f \backslash\{\vzero\}}|a_{\vk}(f)| \exp \left[2\pi \I  \left(\sum_{i\in\sS_{\vk_j}}k_ib_i+\theta_k(f)\right)\right]\notag\\
		=&\sum_{\vk\in \sK_f \backslash\{\vzero\}}\frac{|a_{\vk}(f)|(1+(2\pi)^2|\vk|_1^2)}{Z_f}\cdot\frac{Z_f}{1+(2\pi)^2|\vk|_1^2} \cos \left[2\pi  \left(\sum_{i\in\sS_{\vk_j}}k_ib_i+\theta_k(f)\right)\right],\label{measure}
	\end{align}
	where $\theta_{\vk}(f)=\frac{1}{2\pi}\arg a_{\vk}(f)=\frac{1}{2\pi}\tan^{-1}\frac{\Im a_{\vk}(f)}{\Re a_{\vk}(f)}$, and
	\begin{equation}
		Z_f:=\sum_{\vk\in \sK_f \backslash\{\vzero\}}{|a_{\vk}(f)|(1+(2\pi)^2|\vk|_1^2)}\le \|f\|_{\fB}.
	\end{equation}
	The last equality is due to the fact that $f(v)-a_{\vzero}(f)$ is a real functional.
	
	Denote \begin{align}
		g(v,\vk)&:=\frac{Z_f}{1+(2\pi)^2|\vk|_1^2} \cos \left[2\pi \left(\sum_{i\in\sS_{\vk_j}}k_ib_i+\theta(f)\right)\right]\label{in set}\\&=\frac{Z_f}{1+(2\pi)^2|\vk|_1^2}\frac{1}{2}\left(e^{2\pi\I\theta_k(f)}e_{\vk}(v)+e^{-2\pi\I\theta_k(f)}e_{-\vk}(v)\right).
	\end{align}
	Therefore, \begin{align}f(v)-a_{\vzero}(f)&=\sum_{\vk\in \sK_f \backslash\{\vzero\}}\frac{|a_{\vk}(f)|(1+(2\pi)^2|\vk|_1^2)}{Z_f}g(v,\vk),\label{close}\\\|g(v,\vk)\|_{\fH}&=Z_f\sqrt{\frac{1}{2+8\pi^2|\vk|_1^2}}\le Z_f\le\|f\|_{\fB}.\label{bounded norm}
	\end{align}
	
	Due to Eq.~(\ref{in set}), we know $g(v,\vk)\in\fG_{\sK_f \backslash\{\vzero\},f}$.
	
	Now  based on  Eq.~(\ref{close}), we will prove that $f(v)-a_{\vzero}(f)$ is in the closure of the convex hull of the set $\fG_{\sK_f \backslash\{\vzero\},f}$ in Hilbert space $\fH[L_\text{bound}(\Omega)]$. First of all, $f-a_{\vzero}(f)\in\fB[L_\text{bound}(\Omega)]\subset\fH[L_\text{bound}(\Omega)]$ due to Proposition \ref{into}.
	
	Define a random variable $\vk^*$:
	\begin{equation}
		P(\vk^*=\vk)=\frac{|a_{\vk}(f)|(1+(2\pi)^2|\vk|_1^2)}{Z_f}, ~\vk\in \sK_f \backslash\{\vzero\}.
	\end{equation}	
	Hence $\rmE[g(v,\vk^*)]=f(v)-a_{\vzero}(f)$ due to Eq.~(\ref{close}). For any integer $m$, let $\{\vk_j^*\}_{j=1}^m$ be an i.i.d. random variable sequence with the same distribution as $\vk^*$. From Eq.~(\ref{bounded norm}), we know that
	\begin{align}			\rmE\left\|f(v)-a_{\vzero}(f)-\frac{1}{m}\sum_{j=1}^mg(v,\vk_j^*)\right\|^2_{\fH}=\frac{\operatorname{Var}\left[g(v,\vk^*)\right]}{m}\le \frac{\rmE\|g(v,\vk^*)\|_{\fH}^2}{m}\le \frac{\|f\|_{\fB}^2}{m},
	\end{align}
	where the variance is defined in the Hilbert space $\fH[L_\text{bound}(\Omega)]$.
	
	Therefore, by the pigeonhole principle, there exist $\{\vk_j^{**}\}_{j=1}^m\subset\sK_f \backslash\{\vzero\}$, such that
	\begin{equation}
		\left\|f(v)-a_{\vzero}(f)-\frac{1}{m}\sum_{j=1}^mg(v,\vk_j^{**})\right\|^2_{\fH}\le \frac{\|f\|_{\fB}^2}{m}.
	\end{equation}	
	Here	$\frac{1}{m}\sum_{j=1}^mg(v,\vk_j^{**})$ is a convex combination of elements in $\fG_{\sK_f \backslash\{\vzero\},f}$. As $m\to +\infty$, such obtained convex combinations converge to $f(v)-a_{\vzero}(f)$ in $\fH[L_\text{bound}(\Omega)]$. Hence, $f(v)-a_{\vzero}(f)$ is in the closure of the convex hull of the set $\fG_{\sK_f \backslash\{\vzero\},f}$ in the Hilbert space $\fH[L_\text{bound}(\Omega)]$.
\end{proof}

\begin{lem}\label{relu}
	Each element in $\bar{c}+\fG_{\sK_f \backslash\{\vzero\},f}$, where $|\bar{c}|\le \|f\|_{\fB}$,  is in the closure of the convex hull of the set $\fG_{\text{ReLU},f}$  defined in Eq.~\eqref{eqn:GReLU}.
\end{lem}
\begin{proof}[Proof of Lemma \ref{relu}]
	The conclusion comes from the fact that a convex linear combination of ReLU functions can approximate the cosine functions well, whose   proof can be found in \cite[Proposition 19]{lu2021priori}. In fact, they showed that each cosine function in $\fG_{\sK_f \backslash\{\vzero\},f}$ is in  the convex hull of \begin{align}
		&\fG_{\text{ReLU},f}:=\notag\\&\left\{g\left(\sum_{i=1}^{\infty}b_i\Phi_i\right)=c+\gamma \text{ReLU} \left(\sum_{i\in\sS_{\vk_j}}w_ib_i-t\right):4|c|,|\gamma| \leq 4\|f\|_{\fB}, |\vw|_1=1,|t|\le 1,\vk\in\sK_f\right\}.
	\end{align}
\end{proof}

\begin{prop}\label{same norm}
	Suppose that $f\in \fH[L_\text{bound}(\Omega)]$ is a finite dimensional functional based on some $\sD_j$, where recall that $|\sD_j|<\infty$, i.e.,
	\begin{equation}
		f\left(\sum_{i=1}^\infty b_i\Phi_i\right)= f\left(\sum_{i\in\sD_j} b_{i}\Phi_{i}\right)
	\end{equation} for all $-\frac{1}{2}< b_i<\frac{1}{2}$. We have
	\begin{equation}
		\|f\|_{\fH}= \|g_{f}\|_{H^1\left(\left(-\frac{1}{2},\frac{1}{2}\right)^{c_j}\right)},
	\end{equation} where  $c_j:=|\sD_j|<\infty$, and $g_{f}$ is a $c_j$ dimensional function defined by
	\begin{equation}
		g_f(b_{n_1},b_{n_2},\cdots,b_{n_{c_j}}):=f\left(\sum_{i=1}^{c_j} b_{n_i}\Phi_{n_i}\right), n_i\in\sD_j.
	\end{equation}
	Here $H^1\left(\left(-\frac{1}{2},\frac{1}{2}\right)^{c_j}\right)$ is the $H^1$-Sobolev space on $\left(-\frac{1}{2},\frac{1}{2}\right)^{c_j}$.
\end{prop}
\begin{proof}[Proof of Proposition \ref{same norm}]
	Due to Eq.~(\ref{Hilbert}), we have $\|f\|_{\fH}^2=\sum_{\vk \in \sK}(1+(2\pi)^2|\vk|^{2})|a_{\vk}(f)|^2$.
	
	For $\vk\notin\sK_{\sD_j}$ we have \begin{align}
		a_{\vk}(f)=0.
	\end{align}
	Hence we have \begin{equation}
		\|f\|_{\fH}^2=\sum_{\vk \in \sK_{\sD_j}}(1+(2\pi)^2|\vk|^{2})|a_{\vk}(f)|^2.
	\end{equation}	
	Furthermore, $\|g_{f}\|^2_{H^1\left(\left(-\frac{1}{2},\frac{1}{2}\right)^{c_j}\right)}=\sum_{\vk \in \sK_{\sD_j}}(1+(2\pi)^2|\vk|^{2})|a_{\vk}(f)|^2$ by direct calculations.
\end{proof}

Combining Proposition \ref{same norm}, Lemmas \ref{convex}, \ref{convex1} and \ref{relu}, we can prove Theorem \ref{main}.

\begin{proof}[Proof of Theorem \ref{main}]
	Due to Proposition \ref{same norm}, for each element  $g\in\fG_{\text{ReLU},f}$, i.e.,
	$$g\left(\sum_{i=1}^\infty b_i\Phi_i\right)=c+\gamma \text{ReLU} \left(\sum_{i\in\sS_{\vk_j}}w_ib_i-t\right),$$
	we have
	\begin{align}
		\|g\|_{\fH}^2=\left\|c+\gamma \text{ReLU} \left(\sum_{i\in\sS_{\vk_j}}w_ib_i-t\right)\right\|^2_{H^1\left(\left(-\frac{1}{2},\frac{1}{2}\right)^{|\sS_{\vk_j}|}\right)}
	\end{align}
	where the norm $\|~\cdot~\|_{H^1\left(\left(-\frac{1}{2},\frac{1}{2}\right)^{|\sS_{\vk_j}|}\right)}$ is the $H^1$--Sobolev space norm for functions of $b_1,b_2,\cdots,b_{N_{\vk}}$. By direct calculations, we obtain that
	\begin{equation}
		\left\|c+\gamma \text{ReLU} \left(\sum_{i\in\sS_{\vk_j}}w_ia_i-t\right)\right\|^2_{H^1\left(\left(-\frac{1}{2},\frac{1}{2}\right)^{|\sS_{\vk_j}|}\right)}\le\left(c+\frac{3}{2}\gamma\right)^2+\gamma^2\le 80\|f\|_{\fB}^2.
	\end{equation}
	
	Therefore, $\fG_{\text{ReLU},f}$ is bounded by $4\sqrt{5}\|f\|_{\fB}$. Due to Lemmas \ref{convex1} and \ref{relu}, $f$ is in the closure of the convex hull of the set $\fG_{\text{ReLU},f}$. Due to Lemma \ref{convex}, we finish the proof.
\end{proof}

\subsection{Proof of Theorem \ref{cutoff1}}
\begin{proof}[Proof of Theorem \ref{cutoff1}]
	Consider any $v=\sum_{i=1}^\infty b_i\Phi_i(\vx)\in L_\text{cut}(\Omega)$.
	Without loss of generality, suppose $N_{f}\ge N$ for all $\vk_j$ in $f_m$.
	We have
	\begin{align}
		f_m\left(\sum_{i=1}^\infty b_i\Phi_i(\vx)\right)&=c+\sum_{j=1}^m\gamma_j \text{ReLU} \left(\sum_{i=1}^{N_{f}}w_{ij}b_i-t_j\right)\notag\\&=c+\sum_{j=1}^m\gamma_j \text{ReLU} \left(\sum_{i=1}^{N}w_{ij}b_i-t_j+\sum_{i=N+1}^{N_{f}}w_{ij}b_i\right)\notag\\&:=c+\sum_{j=1}^m\gamma_j \text{ReLU} \left(\sum_{i=1}^{N}w_{ij}b_i-t_j+\varepsilon_j\right)
	\end{align}
	where \begin{equation}
		|\varepsilon_j|=\left|\sum_{i=N+1}^{N_{f}}w_{ij}b_i\right|\le|\vomega_j|_1\delta=\delta.
	\end{equation}
	
	Define
	\begin{equation}
		f_m^*\left(\sum_{i=1}^\infty b_i\Phi_i(\vx)\right)=c+\sum_{j=1}^m\gamma_j \text{ReLU} \left(\sum_{i=1}^{N}w_{ij}b_i-t_j\right)\in\fG^*_\text{ReLU,m,f}.
	\end{equation}
	Since ReLU is a Lipschitz continuous function in $\sR$ with  Lipschitz constant $1$ and $\sum_{j=1}^m|\gamma_j| \leq 4\|f\|_{\fB}$, we have
	\begin{equation}
		\left|f^*_m\left(\sum_{i=1}^\infty b_i\Phi_i(\vx)\right)-f_m\left(\sum_{i=1}^\infty b_i\Phi_i(\vx)\right)\right|\le 4\|f\|_{\fB}\delta.
	\end{equation}
	Thus the first inequality in Theorem \ref{cutoff1} is  proved.
	
	For the second inequality in Theorem \ref{cutoff1}, without loss of generality, we suppose $w_{ij}\not= 0$ for all $w_{ij}$ in $f_m$.		
	Denote \begin{align}
		{f}^{**}_m\left(\sum_{i=1}^\infty b_i\Phi_i(\vx)\right)&=c+\sum_{j\in\sS}\gamma_j \text{ReLU} \left(\sum_{i=1}^{N_{f}}w_{ij}b_i-t_j\right)\in\fG^*_\text{ReLU,m,f},\\
		\bar{f}_m\left(\sum_{i=1}^\infty b_i\Phi_i(\vx)\right)&=\sum_{j\notin\sS}\gamma_j \text{ReLU} \left(\sum_{i=1}^{N_{f}}w_{ij}b_i-t_j\right),\notag
	\end{align} where $\sS:=\{j:N_{f}\leq N\}$.
	Therefore,
	\begin{align}
		\|f^{**}_m-f_m\|_{\fH}&=\|\bar{f}_m\|_{\fH}\notag\\
		&\leq \sum_{j\notin\sS}\left\|\gamma_j \text{ReLU} \left(\sum_{i=1}^{N_{f}}w_{ij}b_i-t_j\right)\right\|_{H^1\left(\left(-\frac{1}{2},\frac{1}{2}\right)^N\times (-\delta,\delta)^{N_{f}-N} \right)}\notag\\
		&\le \sum_{j\notin\sS} \sqrt{\left(\frac{9}{4}\gamma_j^2+\gamma_j^2\right)(2\delta)^{N_{f}-N}}\notag\\
		&\le \sum_{j\notin\sS} \frac{\sqrt{13}}{2}|\gamma_j|(2\delta)^{\frac{N_{f}-N}{2}}\notag\\&\le 2\sqrt{13}\|f\|_{\fB}(2\delta)^{1/2}.
	\end{align}
	Here the second inequality is due to Proposition \ref{same norm} in \ref{Asec:proof-main} (with the redefined Fourier coefficients in \eqref{revisit1}).
	The last inequality is due to the condition $\sum_{j=1}^m|\gamma_j| \leq 4\|f\|_{\fB}$ in the definition of $\fG^*_\text{ReLU,m,f}$.
\end{proof}
\section{Applications of the Theorems on Approximation Functionals and Solving PDEs by Neural Networks}\label{sec:PDEs}

In this section, we further discuss how to apply the obtained theorems (Theorem \ref{main} and \ref{cutoff1}) on the approximation of functionals by neural networks, including the application to solving PDEs by neural networks.

\subsection{Applications of the theorems on approximation of functionals by neural networks}\label{sec:app-D1}

Here we discuss how to develop  neural networks for the learning of functionals without curse of dimensionality based on the
error estimates obtained in Theorems \ref{main} and \ref{cutoff1},  in addition to that given in Sec.~\ref{sec:sec3}.

%

\textbf{(i) Theorems \ref{main} for functionals with special structure}:
Consider the case when the functional $f$ in $\fB[L_\text{bound}(\Omega)]$ has the special structure with
\[\sD=\{\{1\},\{2\},\cdots,\{n\},\cdots\},\]
such as linear functionals and the gradient energy functionals as given in Examples \ref{exam1} and \ref{exam2}.
In this case, in $\fG_\text{ReLU,m,f}$ in Theorems \ref{main}, for any $\vk\in\sK_f=\{\vk\in\sK: a_{\vk}(f)\not=0\}$,  it has only one nonzero component by the definition of $a_{\vk}(f)$ in Eq.~\eqref{coefficient}. Therefore,
the $O(1/\sqrt{m})$ approximation given by Theorem \ref{main} in this case takes the following simple form that contains only $3m+1$ parameters:
\begin{equation}f_m\left(\sum_{i=1}^\infty b_i\Phi_i(\vx)\right)=c+\sum_{j=1}^m\gamma_j \text{ReLU} \left(w_{j}b_{n_j}-t_j\right),\label{less parameter}\end{equation}
where $n_j\leq N_{\vk_j}$ is the index of the only nonzero component of $a_{\vk_j}(f)$.


However, such simplification cannot be employed directly in the training of the neural network. In fact, as discussed in (i), when we set up the neural network,  we do not know the index $n_j$, $1\leq j\leq m$, and in this case, we still need to choose an $N\geq  n_j$, $1\leq j\leq m$. In stead of using Eq.~\eqref{enough}, we establish the neural network using the following formula:
\begin{equation}f_m\left(\sum_{i=1}^\infty b_i\Phi_i(\vx)\right)=c+\sum_{j=1}^m\sum_{i=1}^N\gamma_{ij} \text{ReLU} \left(w_{ij}b_{i}-t_{ij}\right),\label{more accuary}
\end{equation}
which includes the desired approximation in Eq.~(\ref{less parameter}).
Eq.~(\ref{more accuary}) has the same order of number parameters, $O(mN)$, as that of Eq.~(\ref{enough}), and
the approximation error of Eq.~(\ref{more accuary}) can reach $O(1/\sqrt{mN})$ by Theorem \ref{main}. This $O(1/\sqrt{mN})$ error is improved significantly from the original approximation error in  Eq.~(\ref{less parameter}) and that in Eq.~(\ref{enough}) for a more general class of functionals discussed in (i), both of which are $O(1/\sqrt{m})$.

\textbf{(ii) For Theorem \ref{cutoff1}}: When we can reduce the functional to a finite dimensional domain such as $L_\text{cut}(\Omega)$, the parameter $N$ in $L_\text{cut}(\Omega)$ is given and can be regarded as the dimension of $L_\text{cut}(\Omega)$. When we set up the neural network,
the parameter $N$ is directly given by $L_\text{cut}(\Omega)$, instead of the unknown $N$ that has to be chosen large enough in the neural network setup based on Theorems \ref{main} discussed above.
In this case, we can also use $f_m$ in Eq.~\eqref{enough} with this given parameter $N$.
By Theorem \ref{cutoff1}, this $f_m$ can approximate $f\in\fB[L_\text{cut}(\Omega)]$ well in $\fH[L_\text{cut}(\Omega)]$ with error $O(1/\sqrt{m})+O(\sqrt{\delta})$, where $\delta$ is defined in $L_\text{cut}(\Omega)$. The number of parameters in Eq.~(\ref{enough}) is 
$O(mN)$. Therefore, our method  does not suffer from curse of dimensionality in this case either.

\subsection{Method for solving PDEs by neural networks based on the obtained theorems on approximation of functionals}

We can apply Theorem \ref{main} or \ref{cutoff1} to build a neural network to solve a PDE problem at a given point. The follow brief discussion is based on the application of Theorem \ref{main}.

Consider following PDE boundary value problem:
\begin{equation}
	\begin{cases}\fL u=g & \text { in } K_1\\  u=0 & \text { on } \partial K_1.\end{cases}\label{PDEpoint}
\end{equation}
where $K_1=(0,1)^d$ and $g\in L_\text{bound}(L_2(K_1))$.
Here a
basis of $K_1$ is $\{\Phi_{i}\}_{i\in\sN}=\{\exp(2\pi\I \vp\cdot\vx)\}_{\vp\in\sZ^{d}}$.
We want to find the solution at a point $\vx_0\in K_1$, i.e., $u(\vx_0,g)$.
This defines a functional whose input is $g\in L^2(K_1)$ and output is $u(\vx_0,g)\in\sR$,
and we    denote this functional as $\fL^{-1}_{x_0}$.
Assume that $\fL^{-1}_{x_0}\in\ \fB[L_\text{bound}(L_2(K_1))]$. Now we approximation $\fL^{-1}_{x_0}$ by the elements in $\fG_\text{ReLU,m,f}$.

Consider the given date set $\{g_s(\vx),u_s(\vx_0)\}_{s=1}^M$, where $u_s(\vx_0):=u(x_0,g_s)$.
We can calculate that $b_{si}=\int_{K_1}g_s(\vx)\Phi_i(\vx)\,\D x$.
The input data set for the neural network is $\{\{b_{si}\}_{i=1}^{N},u_s(\vx_0)\}_{s=1}^M$.
We build a two layer network with $m$ nodes in the hidden layer based on the following form, with some  number $N$ (aiming at $N\geq N_f$):
\begin{equation}
	u^*_s=c+\sum_{j=1}^m\gamma_j \text{ReLU} \left(\sum_{i=1}^{N}w_{ij}b_{si}-t_j\right).
\end{equation}
Here  $u^*_s$ is the output for the value $u_s(\vx_0)$.
We then learn the coefficients $\vtheta:=\{c,\gamma_j,w_{ij},t_j\}$ by the Loss function:
$\fR(\vtheta)=\frac{1}{M}\sum_{s=1}^M|u^*_s-u_s(\vx_0)|^2$. 
Here $\fR(\vtheta)$ can be regarded as an approximation of the norm $\|f-f_m\|^2_{\fH_0}$ $(\|f-f_m\|^2_{\fH_0}\le\|f-f_m\|^2_{\fH})$ by Monte Carlo sampling. Rigorous analysis of the error of such Monte Carlo sampling in $\fH_s$ will be explored in the future work.

Furthermore, based on this method, we can directly obtain an approximate of the solution of the PDE boundary value problem (\ref{PDEpoint}) as follows. For the date set $\{g_s(\vx),u_s(\vy)\}_{s=1}^M$:

\textbf{(i)}: Denote the girds of $K_1$ as $\{\vy_q\}_{q=1}^Q$. We obtain the $Q$ function-point sets: $\{\sT_q\}_{q=1}^Q$, where $\sT_q:=\{g_i(\vx),u_i(\vy_q)\}_{s=1}^M$.

\textbf{(ii)}: For each $\sT_q$, we learn a functional $f_q\in\fG_\text{ReLU,m}$ from the neural network to approximate the solution of the problem (\ref{PDEpoint}) at $\vy_q$. We obtain functional sets $\{f_q\}_{q=1}^{Q}$.


Here we obtained the approximation of operator $\fL^{-1}$ in the PDE boundary value problem (\ref{PDEpoint}) at several points.

\section{Conclusions and discussion}
In this paper, we establish a neural network to approximate functionals without curse of dimensionality based on the method of Barron space.
The method is developed by defining a Fourier-type series on the infinite-dimensional space of functionals and  the associated spectral Barron space  of functionals. The approximation error of the neural network is $O(1/\sqrt{m})$ where $m$ is the size of networks, which overcomes the curse of dimensionality. The number of parameters and the network structure in our method only depends on the functional,  thus it is not sensitive to the input functions in training.

The proposed method for approximation of functionals without curse of dimensionality can be employed in learning functionals, such as linear functionals and energy functionals in science and engineering fields. It can also be used to solve PDE problems by neural networks at one or a few given points. 
This method provides a basis for the further development of methods for learning operators and analysis of properties (e.g., stability~\cite{Stability}) of neural networks for functionals and operators.

\section*{Acknowledgments}
This work was supported by HKUST IEG19SC04   and the Project of Hetao Shenzhen-HKUST Innovation Cooperation Zone HZQB-KCZYB-2020083.

			\bibliographystyle{unsrtnat}
			\bibliography{references}


		\newpage	
		\appendix

		\section{Theorems on Approximations of Functions, Functionals and Operators}
		
		In this subsection of Appendix, we summarize some available theorems on the approximations of functions, functionals and operators, which are used in our proofs.
		
		\begin{thm}[\cite{mhaskar1996neural}]\label{function}
			Suppose $\sigma$ is a is a continuous non-polynomial function and $K$ is a compact in $\sR^d$, then there are  positive integers $p$,  constants $w_k, \zeta_{k}$ for $k=1, \ldots, p$ and continuous linear functionals $c_k:W^{q}_{r}(K)\to \sR$ such that for any $v\in W^{q}_{r}(K)$,
			\begin{equation}
				\left\|v(\vx)-\sum_{k=1}^{p} c_{k}(v) \sigma\left(\vw_{k}\cdot \vx+\zeta_k\right)\right\|_{L_{q}(K)} \leq c p^{-r / d}\|v\|_{W_{r}^{q}(K)},
			\end{equation}
			where $W_{r}^{q}(K)$ is the set of function in $L_q(K)$ with finite Sobolev norms \begin{equation}
				\|g\|_{W_{r}^{q}(K)}:=\sum_{0 \leq |\vj| \leq r}\left\|D^{\vj} g\right\|_{L_{q}(K)}.
			\end{equation}
		\end{thm}

		\begin{thm}[\cite{chen1993approximations}]\label{functional}
			Suppose $\sigma$ is a is a continuous non-polynomial function, $U$ is a compact set in $C([a,b]^{d_1})$. $f$ is a continuous functionals defined on $U$. Then for any $\varepsilon > 0$, there are positive integers $n, m$, constants $c_i, \xi_{i j}, \theta_{i} \in \mathbb{R}$ for $i=1, \ldots, n, j=1, \ldots, m$ such that
			\begin{equation}
				\left|f(v)-\sum_{i=1}^{n} c_{i} \sigma\left(\sum_{j=1}^{m} \xi_{i j} v\left(\vx_{j}\right)+\theta_{i}\right)\right|\le \varepsilon
			\end{equation}
			holds for all $v\in U$.
		\end{thm}

		\begin{thm}[\cite{chen1995universal}]\label{operator}
			Suppose $\sigma$ is a is a continuous non-polynomial function, $K_1= [0,1]^{d_1}$, $K_2= [0,1]^{d_2}$, $V$ is a compact set in $C(K_1)$ and $G$ is a nonlinear continuous operator, which maps $V$ into $C(K_2)$. Then for any $\varepsilon > 0$, there are positive integers $n, p, m$, constants $c_i^{k}, \xi_{i j}^{k}, \theta_{i}^{k}, \zeta_{k} \in \mathbb{R}, \vw_{k} \in \mathbb{R}^{d_2}, \vx_{j} \in K_{1}$ for $i=1, \ldots, n, k=1, \ldots, p, j=1, \ldots, m$ such that
			\begin{equation}
				\left|G(v)(\vy)-\sum_{k=1}^{p} \sum_{i=1}^{n} c_{i}^{k} \sigma\left(\sum_{j=1}^{m} \xi_{i j}^{k} v\left(\vx_{j}\right)+\theta_{i}^{k}\right){\sigma\left(\vw_{k} \cdot \vy+\zeta_{k}\right)}\right|\le \varepsilon
			\end{equation}
			holds for all $v\in V$ and $\vy\in K_2$.
		\end{thm}

					\end{document}